\documentclass[11pt,oneside, a4paper]{amsart}
\usepackage{amsmath,amssymb,amsfonts} 
\usepackage{graphics}                 
\usepackage{color}                    
\usepackage{hyperref}                 

\usepackage[latin1]{inputenc}

\usepackage[all]{xy}
\usepackage{enumerate}
\title[Iterative differential Galois theory]{Iterative differential Galois theory\\ in positive characteristic: \\A model theoretic approach}
\author{Javier Moreno}
\thanks{This paper contains the results of the author's thesis, which was written under the kind supervision of Anand Pillay. The author would like thank him for his patience, advice and support.} 
\date{June, 2009}

\address{Institut Camille Jordan\\
Universit\'e Claude Bernard Lyon 1\\
43 boulevard du 11 novembre 1918\\
69622 Villeurbanne Cedex, France}
\email{moreno@math.univ-lyon1.fr}
\urladdr{http://math.univ-lyon1.fr/~moreno/}
\subjclass[2000]{Primary 03C98; Secondary 12H05}

\newtheorem{theorem}{Theorem}[section]

\newtheorem{fact}[theorem]{Fact}
\newtheorem{lemma}[theorem]{Lemma}
\newtheorem{corollary}[theorem]{Corollary}
\theoremstyle{definition}
\newtheorem{definition}[theorem]{Definition}
\newtheorem{example}[theorem]{Example}

\def\dd{\partial} 
\def\UU{\mathcal{U}} 
\def\CC{\mathcal{C}} 
\def\AA{\mathcal{A}} 
\def\LL{\mathcal{L}} 

\def\MM{\mathcal{M}} 
\def\NN{\mathcal{N}} 
\def\XX{\mathcal{X}} 
\def\A{\mathbb{A}} 
\def\DD{\mathbb{D}}
\newcommand{\dgen}[1]{\langle #1 \rangle} 
\newcommand{\Imag}{\operatorname{Im}}
\newcommand{\Gal}{\operatorname{Gal}}
\newcommand{\Th}{\operatorname{Th}}
\newcommand{\gal}{\operatorname{gal}}
\newcommand{\tp}{\operatorname{tp}}
\newcommand{\Aut}{\operatorname{Aut}}
\newcommand{\dcl}{\operatorname{dcl}}
\newcommand{\acl}{\operatorname{acl}}

\newcommand{\Cb}{\operatorname{Cb}}
\newcommand{\trdeg}{\operatorname{tr.deg}}

\def\Ind#1#2{#1\setbox0=\hbox{$#1x$}\kern\wd0\hbox to 0pt{\hss$#1\mid$\hss}
\lower.9\ht0\hbox to 0pt{\hss$#1\smile$\hss}\kern\wd0}
\def\ind{\mathop{\mathpalette\Ind{}}}
\def\Notind#1#2{#1\setbox0=\hbox{$#1x$}\kern\wd0\hbox to 0pt{\mathchardef
\nn=12854\hss$#1\nn$\kern1.4\wd0\hss}\hbox to                            
0pt{\hss$#1\mid$\hss}\lower.9\ht0 \hbox to                            
0pt{\hss$#1\smile$\hss}\kern\wd0}

\begin{document}

\begin{abstract}
This paper introduces a natural extension of Kolchin's differential Galois theory to positive characteristic iterative differential fields, generalizing to the non-linear case the iterative Picard-Vessiot theory recently developed by Matzat and van der Put. We use the methods and framework provided by the model theory of iterative differential fields. We offer a definition of strongly normal extension of iterative differential fields, and then prove that these extensions have good Galois theory and that a $G$-primitive element theorem holds. In addition, making use of the basic theory of arc spaces of algebraic groups, we define iterative logarithmic equations, finally proving that our strongly normal extensions are Galois extensions for these equations. 
\end{abstract}

\maketitle


\section{Introduction}
Differential Galois theory is the study of extensions of differential fields with well-behaved automorphism groups. In contrast to classic algebraic Galois theory, the automorphism groups of differential Galois extensions turn out to be algebraic groups or, in more general settings, differential algebraic groups. The origins of differential Galois theory trace back to the works of Picard and Vessiot in the 1890s studying linear differential equations, but it was Kolchin (following the work of Ritt) who gave the first systematic and modern account of the subject in his seminal book \cite{Kolchin1}. There he introduced, among other things, the notion of a strongly normal extension of differential fields as a non-linear generalization of the classic extensions discovered by Picard and Vessiot. Modern treatments of Kolchin's theory have been offered in recent years by Magid \cite{Magid}, Kovacic \cite{Kovacic} and Umemura \cite{Umemura}, among others. The current general reference on the subject is the book by Singer and van der Put \cite{vanderPut-Singer}.

For basic technical reasons regarding the nature of constants in positive characteristic, most of the work on differential Galois theory has been restricted to the context of characteristic zero differential fields, but several attempts have been made to work around this obstacle \cite{andre}\cite{vanderPut-Singer}. One of the most popular strategies is replacing the notion of a derivative for that of an iterative (Hasse-Schmidt) derivation. This approach was extensively studied (but with partial success) by Okugawa \cite{Okugawa1}\cite{Okugawa2} and Shikishima-Tsuji \cite{Shikishima}, and more recently by Matzat and van der Put \cite{Matzat}, who developed a full Picard-Vessiot theory for positive characteristic iterative differential fields with the use of the theory of torsors. 

Differential Galois theory and model theory have had a long history. It all started with Poizat's paper {\em Une théorie de Galois imaginaire} \cite{Poizat3} suggesting that Kolchin's main results could be obtained as a consequence of the $\omega$-stability of the theory of differentially closed fields and some work by Zilber and Hrushovski on the definability of automorphism groups \cite{Hrushovski}\cite{Zilber}. After this, several people including Marker, Pillay and Sokolovi\'c \cite{Anand-Marker}\cite{Anand3.5}\cite{Anand3}\cite{Anand2.5}\cite{Anand-Sokolovic} have contributed to the development and even generalization of Kolchin's results making use of these abstract tools from model theory. 

Positive characteristic differential Galois theory, however, remained out of reach of model theory until Messmer, Wood and Ziegler proved (strongly based on Delon's work on separably closed fields of positive characteristic \cite{Delon}) that the theory of fields equipped with stacks of commuting iterative derivations has a stable model companion with quantifier elimination and elimination of imaginaries \cite{Messmer-Wood}\cite{Ziegler}. After this, adapting some of his previous results in characteristic zero, and working along the lines of what Hrushovski suggested in \cite{Udi}, Pillay found alternative proofs of existence and uniqueness of iterative Picard-Vessiot extensions (results already proved by Matzat and van der Put) using now model-theoretic techniques \cite{Anand2}. 

Following Pillay, this article introduces a theory of strongly normal extensions for iterative differential fields of positive characteristic. These extensions generalize the Picard-Vessiot extensions developed by Matzat and Van Der Put as well as the strongly normal extensions proposed by Okugawa. Our results depend on the model theory of iterative differential fields. 

After introducing in section \ref{Section:Uno} the basic definitions and model theory of iterative differential fields of positive characteristic, in section \ref{Section:Dos} we define what we mean by an iterative strongly normal extension (Definition \ref{def:isne}). Then we prove that the Galois group of these extensions is isomorphic to the constant-rational points of an algebraic group defined over the constants of the base field (Theorem \ref{fact:definability}), we have good Galois correspondence (Theorem \ref{th:galcorr}), and also what Kolchin called a $G$-primitive element theorem (Theorem \ref{thm:pet}).

In section \ref{Section:Tres} we start all over again, this time from the perspective of (logarithmic) differential equations. This requires us to make an overview of the notion of the arc bundle of an algebraic variety (Definition \ref{def:arcs}) and introduce what should be our logarithmic derivation (Definition \ref{def:logder}). Once there we define what we mean by an iterative differential Galois extension for a given logarithmic differential equation (Definition \ref{def:idge}) and prove that these extensions exist and are unique modulo isomorphism (Theorem \ref{th:existanduniq}). 

Finally, in section \ref{Section:cuatro}, we show that, under certain hypothesis on the base field, iterative strongly normal extensions and iterative differential Galois extensions are just two faces of the same notion (Theorems \ref{th:galextstrnor} and \ref{th:equiv}). 

We assume that the reader has working knowledge of the fundamentals of geometric model theory, and a fair understanding of the terminology of varieties from algebraic geometry and basic differential algebra. 

\section{Iterative differential algebra: Model theory and practice}\label{Section:Uno} 

The aim of this section is offering a brief introduction to the basic model theory of iterative derivations. 

\begin{definition}Let $R$ be an arbitrary ring. A sequence of maps $\dd=(\dd_{i}\colon R\to R)_{i\in \omega}$ is called a \textbf{Hasse-Schmidt derivation} if $\dd_{0}=id_{R}$ and the map 
\[\DD_{\dd}\colon R\to R[[\epsilon]]\colon a\mapsto \sum_{i=0}^{\infty}\dd_{i}(a)\epsilon^{i}\] is a ring homomorphism. 

If, additionally, for any $i, j\in \omega$ we have that $\dd_{i}\circ\dd_{j}=\binom{i+j}{i}\dd_{i+j}$, we say that $\dd$ is an \textbf{iterative Hasse-Shmidt derivation} or simply an \textbf{iterative derivation}. 

A ring (field) R equipped with an iterative derivation $\dd$ is what we call an \textbf{iterative differential ring (field)} or ID-ring (field) and its \textbf{ring (field) of constants}, $C_{R}$, is defined as the set where all the $\dd_{i}$ vanish.

We say that an ID-field $F$ is \textbf{non-trivial} if $\dd_{1}|_{F}\not\equiv 0$. \end{definition}

Let $IDF_{p}$ be the first-order theory of fields of characteristic $p>0$ equipped with an iterative derivation $\dd$. The language we will consider is that of fields expanded with a sequence $(\dd_{i})_{i<\omega}$ of unary function symbols. This theory has a model companion, $SCH_{p}$,  the theory of separably closed ID-fields, $K$, of characteristic $p$, degree of imperfection $1$ (i.e. $[K:K^{p}]=p$) and $K^{p}=\{x\in K\colon \dd_{1}(x)=0\}$. 

Given  a model $K$ of $SCH_{p}$, we can see that $C_{F}=K^{p^{\infty}}$, an algebraically closed field. This theory is, in some sense, just another version of $SCF_{p,1}$, the theory of separably closed fields of characteristic $p$ and degree of imperfection $1$:
 
\begin{fact}\label{fact:SCH-SCF} Once a $p$-basis is fixed, every model of $SCF_{p,1}$ can be expanded to a model of $SCH_{p}$ and, additionally, any highly enough saturated model of $SCH_{p}$ can be canonically equipped with a $p$-basis and its corresponding $\lambda$-functions are quantifier-free definable.\end{fact}
\begin{proof}See \cite{Ziegler}.\end{proof}

And, as a consequence of this,

\begin{fact}\label{fact:SCH}
$SCH_{p}$ is stable (non-superstable) and has quantifier elimination and elimination of imaginaries. 
\end{fact}
\begin{proof}See \cite{Ziegler}.\end{proof}

Since $SCH_{p}$ is stable, we may let $(\UU,\dd)$ be a saturated model of $SCH_{p}$ of {\em large} cardinality and $\CC$ its field of constants.

As in the case of characteristic zero differentially closed fields, the field of constants of $\UU$ is a {\em pure} algebraically closed field, that is, any definable subset is definable in the language of rings. Note that in this case, though, $\CC$ is not definable but type-definable:  
\begin{fact}\label{lem:Cpure}
If $Z\subset \UU^{m}$ is definable in $\UU^{m}$ over $A$, then $Z\cap \CC^{m}$ is definable in $(\CC, +,\ \cdot\ )$ over $\dcl(A)\cap\CC$. 
\end{fact}
\begin{proof}By stability and quantifier elimination.\end{proof}

We will now define three different closure operators of algebraic nature. 
\begin{definition}Let $A\subset \UU$.
\begin{itemize}
\item \textbf{The iterative differential closure of $A$}, denoted $\dgen{A}$, will be the iterative differential subfield of $\UU$ generated by the elements of $A$. If $F$ is an $ID$-field and $A$ is a set, by $F\dgen{A}$ we mean $\dgen{FA}$.
\item \textbf{The strict closure of $A$}, denoted $A^{s}$, will be the set obtained after closing $A$ under $p$th-roots. 
\item Finally, \textbf{the relative algebraic closure of $A$}, denoted $A^{a}$, will be the field theoretic algebraic closure of $A$ inside $\UU$. 
\end{itemize}
\end{definition}

Benoist \cite{Benoist} gave useful algebraic characterizations of the model theoretic definable and algebraic closures in $\UU$ in terms of these closure operators:
\begin{fact}
Let $A\subset \UU$. 
\begin{itemize}
\item $\dcl(A)=\dgen{A}^{s}$
\item $\acl(A)=\dgen{A}^{a}$
\end{itemize}
\end{fact}
\begin{proof} See Proposition II.2 and Proposition II.3 of  \cite{Benoist}.
\end{proof}

\section{Iterative strongly normal extensions}\label{Section:Dos}
From now on, let us fix a prime number $p>0$. As in the previous chapter, let $\UU$ be a saturated model of $SCH_{p}$ of large cardinality where any ID-field mentioned is embedded and let $\CC$ the field of constants of $\UU$. 

Let us also assume that $(F, \dd)$ is an iterative differential field with $\dd_{1}|_{F}\not\equiv 0$. 

\subsection{Definition and basic properties}

\begin{definition}\label{def:isne} An extension $(F,\dd)<(K,\dd)$ of non-trivial definably closed ID-fields is said to be \textbf{strongly normal} if the following conditions hold:

\begin{enumerate}
\item $C_{F}=C_{K}$, and $C_{F}$ is algebraically closed;
\item $K=F\dgen{a}^{s} (=\dcl(Fa))$ for some $a=(a_{1},\cdots,a_{m})$;
\item Whenever $\sigma\colon K \hookrightarrow \UU$ is an embedding of
  $K$ into $\UU$ over $F$, then $\sigma(K)\subseteq K\dgen{\CC}$; and finally,
\item $F^{a}\cap K = F\dgen{d}^{s}$ for some $d=(d_{1},\cdots,d_{m})$.
\end{enumerate}

Following Kolchin, the \textbf{Galois group} of the strongly normal extension $K/F$, denoted $\Gal(K/F)$ will be $\Aut_{\dd}(K\dgen{\CC}/F\dgen{\CC})$. The traditional $\Aut_{\dd}(K/F)$ will be denoted instead $\gal(K/F)$.\end{definition}

Our definition of a strongly normal theory does not differ much from the original one due to Kolchin in the characteristic zero case. We require, though, one property that in Kolchin's case is automatic: $\tp(a/F)$ should have finite multiplicity. This is going to be crucial to assure the definability of our Galois group. The extra-condition (4) will do this for us:

\begin{fact}\label{fact:mult} If $K=F\dgen{a}^{s}$ is an iterative strongly normal extension of $F$, then $\tp(a/F)$ has finite multiplicity.\end{fact}
This is a corollary of the following general lemma:
\begin{lemma}
Let $T$ be a stable theory and $\UU$ a highly saturated model of $T$. Then, for
any $a\in \UU$ and $F\subseteq \UU$, we have that $\tp(a/F)$ has finite
multiplicity if and only if there exists a finite tuple $c\in \UU^{eq}$ 
such that $\acl^{eq}(F)\cap\dcl^{eq}(aF)=\dcl^{eq}(Fc)$.
\end{lemma}
\begin{proof}
$\Rightarrow$) Let $p=\tp(a/F)$ and let $p_1,\ldots,p_m$ be the complete
extensions of $p$ to $\acl(F)$, and $p_1=\tp(a/acl(F))$.  The finite equivalence relation theorem provides us with a single $F$-definable finite equivalence relation $E$ distinguishing the extensions of $p$ to $\acl(A)$. Let $c$ be the $E$-class of $a$. By definition $c\in \dcl^{eq}(Fa)$, and, since $E$ is an equivalence relation over $F$ with finitely many classes, $c\in \acl^{eq}(F)$. 

On the other hand, let $b\in \dcl^{eq}(aF)\cap \acl^{eq}(F)$. Thus we have $b=f(a)$ for some
$F$-definable function $f$. Let $\sigma$ be an automorphism of $\UU$ fixing $Fc$. Suppose $\sigma(p_{1})\neq p_{1}$. Then $\sigma(p_{1})=p_{i}$, for some $i\neq 1$ and so $\sigma(c)\neq c$, contradicting the choice of $\sigma$. Thus, the formula $\sigma(b)=f(x)$ is still in $p_1$. But then, $\sigma(b)=f(a)=b$. This implies that $b\in \dcl^{eq}(Fc)$.\\
$\Leftarrow$) Let $d=\Cb(\tp(a/\acl^{eq}(F)))$. We know that $d\subseteq \acl^{eq}(F)$
and it is also clear that $d\subseteq \dcl(Fa)$. Thus $d\subseteq
\dcl^{eq}(Fc)$ for some $c\in \acl^{eq}(F)$. But this implies that $d$ has
finitely many conjugates over $A$, and so $\tp(a/F)$ has finite
multiplicity.
\end{proof}

One consequence of condition (3) in the definition of iterative strongly normal extensions is the fact that $\tp(a/F)$ is {\em internal} to $\CC$: If $\tp(b/F)=\tp(a/F)$ then  $b\in K\dgen{\CC}=\dcl(F,a,\CC)$. The fact that the type has finite multiplicity makes this definability uniform, as we show next. 

\begin{lemma}\label{lem:unifdef}
If $K=F\dgen{a}^{s}$ is an iterative strongly normal extension of $F$, then there exists a function defined over $F$, let us call it $u(\cdot, \cdot)$, such that for every $b$ with $\tp(b/F)=\tp(a/F)$, there is $c \in \CC$ such that $u(a,c)=b$. \end{lemma}
\begin{proof}
When $p=\tp(a/F)$ is stationary, this is a standard fact (see, for instance, Theorem 2.19, p. 37, of \cite{Poizat2}). For the general case, find a function for each complete extension of $p$ to $\acl(F)$ and then glue them together. 
\end{proof}

\subsection{The Galois group is an algebraic group}

This subsection is devoted to prove the following key result:

\begin{theorem}\label{fact:definability} There is an isomorphism of groups \[\mu\colon \Gal(K/F)\to G(\CC),\] where $G$ is algebraic group in $\UU$ defined over $C_{F}$. Furthermore, the action of $\Gal(K/F)$ on $\XX= \tp(a/F)^{\UU}$ is $(F\cup\{a\})$-definable.\end{theorem}

The following lemma and its corollary, both crucial in the proof of this theorem, clarify the nature of $\Gal(K/F)$ and its relation with the associated strongly normal extension. In particular, the lemma tells us that $\gal(K/F)<\Gal(K/F)$.

\begin{lemma}\label{lem:liftingauts}
Any embedding of $K$ into $\UU$ over $F$ can be {\em uniquely} extended to an automorphism of $K\langle\CC\rangle$ fixing $\CC$ pointwise.
\end{lemma}
\begin{proof} It is enough to show that for any $a'\in \UU$ such that $\tp(a/F)=\tp(a'/F)$, we have $\tp(a/F\dgen{\CC})=\tp(a'/F\dgen{\CC})$. To see this, take $\sigma\colon K\to \UU$ an embedding of ID-fields fixing $F$. Since $\tp(a/F)=\tp(\sigma(a)/F)$, our assumption tells us that  $\tp(a/F\dgen{\CC})=\tp(\sigma(a)/F\dgen{\CC})$ and this, by homogeneity of $\UU$ and stability of the theory, provides us with $\bar\sigma$, a $\UU$-automorphism fixing $F\dgen{\CC}$ and taking $a$ to $a'$. Note that $\bar\sigma|_{K\dgen{\CC}}\colon K\dgen{\CC}\to\sigma(K)\dgen{\CC}$. Note also that $\sigma(K)\dgen{\CC}=K\dgen{\CC}$.




In order to prove the claim, consider the infinite tuples $aF$ and $a'F$, where $\tp(a'/F)=\tp(a/F)$. Since $\CC$ is algebraically closed and type-definable over the empty set, $\tp(aF/\CC)$ and $\tp(a'F/\CC)$ are, respectively, the unique nonforking extensions of $\tp(aF/C_F)$ and $\tp(a'F/C_F)$. However, $\tp(aF/C_F)$ and $\tp(a'F/C_F)$ are equal, and, in consequence, the same is true about  $\tp(aF/\CC)$ and $\tp(a'F/\CC)$. \end{proof}

\begin{corollary}\label{cor:prinhomsp}
$\XX$, the set of realisations of $\tp(a/F)$ in $\UU$, is a principal homogeneous space for $\Gal(K/F)$.
\end{corollary}
\begin{proof}
As $\XX\subset K\dgen{\CC}$, then $\Gal(K/F)$ acts on $\XX$. The fact that the action on $\XX$ is transitive and free is a direct consequence of the previous lemma. 

\end{proof}

\begin{proof}[Proof of theorem \ref{fact:definability}]
This is a modified version of the general argument for proving the definability of the binding group. 

Let  $Y=Z/E$ where $Z=\{c\in \CC\ \colon\ u(a,c)\in\XX\}$ and $E$ is an equivalence relation on $Z$ defined by the formula $u(a,x_{1})=u(a,x_{2})$. Because of elimination of imaginaries of the theory of algebraically closed fields and the pureness of $\CC$, $Y$ is a type-definable set in $\CC$ over $dcl(a)\cap \CC\subset C_{F}$. 

For $b\in \XX$ and $d\in Y$, define $f(b,d)=u(b,c)$ with $c\in Y_{0}$ such that $c/E=d$, and note that  for any $b_{1}$ and $b_{2}\in \XX$, there is only one $d\in Y$ such that $f(b_{1},d)=b_{2}$. 

Consider the function $\mu\colon \Gal(K/F) \to Y\colon \sigma\mapsto h(a,\sigma(a))$. Corollary \ref{cor:prinhomsp} then tells us that $\mu$ is a bijection. Endow $Y$ with the group operation induced by $\mu$. This is, let us define $d\cdot d'=\mu(\mu^{-1}(d)\cdot\mu^{-1}(d'))$. Note that this group operation is definable. Moreover, the induced action of $Y$ on $\XX$ turns out to be $F\cup{a}$-definable. 

Finally, the fact that $\CC$ is totally transcendental plus the Weil-Van den Dries-Hrushovski theorem (See Theorem 4.13, p. 84 of \cite{Poizat2}, or \cite{Lou}) tells us that $(Y,\cdot)$ is definably isomorphic to the set of $\CC$-rational points of an algebraic group $G$ defined over $C_F$. Identify $G(\CC)$ and $Y$.
\end{proof} 

In addition, $\mu$ takes $\gal(K/F)$ to the $C_F$-rational
points of $G$.
\begin{fact}$\mu(\gal(K/F))=G(C_F)$\end{fact}
\begin{proof}
First, observe that for any $\sigma\in \Gal(K/F)$, we have that $\sigma(a)\in F\dgen{a,\mu(\sigma)}^{s}$ and $\mu(\sigma)\in F\dgen{a,\sigma(a)}^{s}\cap \CC$. Now, if $\sigma(a)\in K$, then $\mu(\sigma)\in F\dgen{a}^{s}\cap \CC=C_K=C_F$. 

On the other hand, if $\mu(\sigma)\in C_F$, then, $\sigma(a)\in F\dgen{a}^{s}=K$. \end{proof}

\subsection{Scaffolding}

In characteristic zero, the model theoretic approach to differential Galois theory \cite{Anand3} heavily depends on the existence of prime models, a basic consequence of the fact that $DCF_{0}$ is totally transcendental. Since $SCH_p$ is only stable and not even superstable, we need to rely in other tools to deal with the lack of decent {\em differential closure}. Just like in the linear case \cite{Anand2}, the use of a suitable auxiliary structure as a scaffolding to handle the group inside $\UU$ will do the trick. 

\begin{definition}Let $K/F$ be an iterative strongly normal extension, with $K=F\dgen{a}^{s}$. Now define $\MM$ as the two-sorted structure $(\XX,\CC)$, with relations induced by $F$-definable relations in $\UU$. \end{definition}

\begin{fact}\label{fact:biinterp}
Let $\NN$ be the structure whose universe is $\CC$, with relations induced by the $F\dgen{a}^{s}$-definable sets in $\UU$. Then $(\MM, a)$ is bi-interpretable with $\NN$. 
\end{fact}
\begin{proof}
On one hand, any intersection of $\CC$ and a $F\dgen{a}^{s}$-definable set in $\UU$ is, inside $\MM$, an $a$-definable set. The other direction depends on the definability of the Galois group from the previous section. 

As in the proof of theorem \ref{fact:definability}, let us define $Y$ as the quotient of $Z$ by $E$, where $Z$ is the type-definable set $\{c\in \CC\ \colon\ u(a,c)\in\XX\}$ and $E$ is given by the formula $u(a,x_{1})=u(a,x_{2})$. By the proof of theorem \ref{fact:definability} we know that $Y$ is a $C_{F}$-definable set in $\CC$.

Note that $\XX$ and $Y$ are isomorphic. Indeed, for each $b\in \XX$ assign the class $\bar{c}_{b}$ of $c\in\CC$ such that $u(a,c)=b$. This map is one-to-one and onto by construction. Now, suppose that you have $D\subset \MM\cap\XX$ definable in $(\MM,a)$. By definition, $D$ is $F\dgen{a}^{s}$-definable in $\UU\cap \XX$. The map given between $\XX$ and $Y$ is also $F\dgen{a}^{s}$-definable; thus the image of $D$, now inside the quotient set, is also $F\dgen{a}^{s}$-definable. This makes $D$ definable inside $\NN$. 
\end{proof}
Now we can check that, although we are working in a stable, non-superstable theory, the auxiliary structure we built is totally transcendental.
\begin{fact}\label{fact:M}
$\MM$ is saturated and its theory $\Th(\MM)$ has quantifier elimination and is totally transcendental. 
\end{fact}
\begin{proof}
Let $\NN$ be just as in fact \ref{fact:biinterp}. 

Since $\NN$ can be seen as $\CC$ with names for the elements of $C_F$, then it is saturated and totally transcendental. The fact that being totally transcendental and saturation are preserved under taking reducts and interpretability, allow us to conclude that $\MM$ is also saturated and totally transcendental.

Since $\MM$ is saturated, for quantifier elimination it is enough to prove that it is also quantifier-free homogeneous. Indeed, if $d_{1}$ and $d_{2}$ are two finite tuples from $\MM$ with the same quantifier-free type, then, seeing them as tuples from $\UU$, they have the same type over $F$. The homogeneity of $\UU$ then provide us with an automorphism of $\UU$ over $F$ taking one to the other. As it fixes $F$, this function is also an automorphism of $\MM$ when restricted to its domain.\end{proof}

Given that $\Th(\MM)$ is totally transcendental, let $\MM_{0}$ be its prime model over the empty set.  This structure will play the role of the differential closure of $F$.  

\begin{lemma}\label{lem:primeM} \[\MM_{0}\cap \CC=C_{F}\]\end{lemma}
\begin{proof} Let $c\in \MM_{0}\cap \CC$ and consider $p$, the type of $c$ over the empty set in the language of $\MM$. Since $\MM_{0}$ is prime, $p$ is isolated by a formula $\phi(x)$. By lemma \ref{lem:Cpure}, the set that this formula defines inside $\CC$ is also defined by a formula in the language of rings and with parameters in $F\cap \CC=C_{F}$. However, being algebraically closed, $C_{F}$ is an elementary substructure of $\CC$ (in the language of rings), and so $\phi(\CC)\cap C_{F}$ is not empty. This implies that, as  $\phi$ is an isolating formula over $F$, it must be of the form $x=c'$ for some $c'\in C_{F}$.\end{proof}  

The following fact tells us that the whole extension can be somehow interpreted, in a multi-sorted way, inside $\MM_{0}$, as if it were an scaffolding built on its side. 

\begin{lemma}\label{lem:correspM0}
There is a bijection between the set of definably closed subsets of $\MM_{0}^{eq}$ and the set of definably closed $ID$-fields lying between $F$ and $K$.
\end{lemma}
\begin{proof}
Any $d\in \MM_{0}^{eq}$ is of the form $a'/E$ where $E$, by quantifier elimination, is a quantifier free $\emptyset$-definable equivalence relation in $\MM$. This means that, in $\UU$, we have that $E$ is the intersection of $\XX$ and some $F$-definable set $E'$ in $\UU$. By stability of $SCH_{p,1}$, it can be assumed that $E'$ is also an equivalence relation. By elimination of imaginaries in $SCH_{p,1}$, we know that $a'/E'$ is interdefinable over $F$ in $\UU$ with some tuple $e\in \dcl(F,a')$. Note that, since there is $c\in \CC\cap\MM_{0}=C_{F}$ such that $a'=u(a,c)$, then, $e\in \dcl(F,a)=K$. Thus, $d$ is interdefinable over $F$ in $\UU$ with a tuple in $K$. 

Let now $e\in K$. Then, $e=f(a)$ for some $F$-definable function $f$. Let $E(x,y)$ be $f(x)=f(y)$. The restriction of $E$ is $\emptyset$-definable in $\MM_{0}$ and $d=a/E\in \MM_{0}^{eq}$. Clearly, $d$ (seen as an element in $\UU$) is interdefinable over $F$ with $e$. \end{proof}

\subsection{Galois correspondence and a $G$-primitive element theorem}

Let $K$ a strongly normal extension of $F$ and $G$ the algebraic group whose $\CC$-rational points are isomorphic to $\Gal(K/F)$, as provided by theorem \ref{fact:definability}. As in the previous subsection, let $\MM$ the scaffolding built for the extension $K/F$ and $\MM_0$ its prime model over the empty set. 

\begin{definition}Given $L$ a definably closed subfield of $K$ containing $F$, let \[G_L=\{g\in G(\CC)\ :\ g(c)=c \textrm{ for all }c\in L\}.\]\end{definition}

\begin{theorem}\label{th:galcorr} Let $K/F$ be a strongly normal extension of ID-fields. If $L$ is a definably closed intermediate ID-field in $K/F$, then:
\begin{enumerate}[(i)]
\item $K/L$ is strongly normal.
\item $G_L$ is a $C_{F}$-definable subgroup of $G(\CC)$ and is isomorphic to $Gal(K/L)$.
\item The correspondence $L\mapsto G_L$ between intermediate ID-fields and $C_{F}$-definable subgroups of $G$ is an injection. 
\item $L/F$ is strongly normal if and only if $G_L$ is a normal subgroup of $G(\CC)$. In this case, $G(\CC)/G_L\cong Gal(L/F)$.
\end{enumerate}
\end{theorem}

Before proving the theorem, let us observe that definably closed intermediate fields of a strongly normal extension are finitely generated over the base field. More precisely:

\begin{lemma}\label{lem:finitelygen}If $K/F$ is strongly normal and  $L$ is an intermediate definably closed $ID$-field, then $L=F\dgen{b}^{s}$ for some $b=(b_{1},\cdots,b_{m})$. 
\end{lemma}
\begin{proof}
Consider $L$ as a definably closed subset in $\MM_{0}^{eq}$ and let $p=\tp(a/L)$. Since $Th(\MM)$ is totally transcendental, there is a finite tuple $b$ such that $p$ is the unique non-forking extension over $L$ of $\tp(a/Fb)$. This $b$ is the tuple of the canonical bases of each of the finitely many complete extensions of $p$ to $\acl(L)$. We claim $F\dgen{b}^{s}=L$. 

The left to right containment is clear. On the other hand, if $e\in L$, let $g(\cdot)$ an $F$-definable function such that $g(a)=e$. Consider the formula $\phi(x,y)$ defined as $g(x)=y$ and let $d\phi_{x}(y)$ be the $\tp(a/\acl(L))$-definition of $\phi$ over $\dcl(F,b)$. Clearly, $d\phi_{x}(e')$ iff $e=e'$, and so $e\in\dcl(F,b)$. 
\end{proof}

\begin{proof}[Proof of theorem \ref{th:galcorr}]
Let $L=F\dgen{b}^{s}$.

(ii) and (iii) are easy.

For (i), observe that as $K/F$ is strongly normal, conditions (1), (2) (by lemma \ref{lem:finitelygen}) and (3) of the definition of strongly normal extensions immediately hold in $K/L$. Condition (4) requires an explanation:

The correspondence provided by lemma \ref{lem:correspM0} allows us to see $D=\acl(L)\cap K$ as a subset of $\MM_{0}^{eq}$. Consider then, in $\MM$, the type $\tp(a/D)$. Since $K=L\dgen{a}^{s}$, the canonical base of $\tp(a/\acl(L))$ is contained inside $D$. By $\omega$-stability of $Th(\MM)$, we have that $\Cb(\tp(a/\acl(L))$ is interdefinable with $d$, a single finite tuple in $\MM^{eq}$. It is easy to check that $L\dgen{d}^{s}=D$.

Finally, for (iv), suppose that $L/F$ is strongly normal and let $N$ be the normalizer of $\Gal(K/L)$ in $\Gal(K/F)$. As both $\Gal(K/L)$ and $\Gal(K/F)$ are definable, so is $N$. Thus, by part (3) of the present theorem, $N=G_{L'}$ for some $L'$ such that $F<L'<L$. We will prove that $N=\Gal(K/F)$:

Let $\sigma\in \Gal(K/F)$, an automorphism $\tau\in \Gal(K/L)$ and $d\in L$. Since $L/F$ is strongly normal, $\sigma(d)\in L\dgen{\CC}$ and so $\tau\sigma(d)=\sigma(d)$. This implies that $\sigma^{-1}\tau\sigma(d)=d$ and thus we conclude that $\sigma^{-1}\tau\sigma\in\Gal(K/L)$ and, moreover, $\sigma\in N$ as $\tau\in \Gal(K/L)$ is arbitrary. 

Assume now that $G_{L}=\Gal(K/L)$ is a normal subgroup of $\Gal(K/F)$. Conditions (1), (2) and (4) from the definition of strongly normal extensions are clear for $L/F$. We need to prove (3): Let $\sigma \colon L \hookrightarrow \UU$ be an embedding of
  $L$ into $\UU$ over $F$. Because of saturation of $\UU$ and quantifier elimination, there is $\bar{\sigma}$ an embedding of $K$ into $\UU$ over $F$ such that $\sigma=\bar{\sigma}|_{L}$. Since $K/F$ is strongly normal, $\bar{\sigma}$ can be seen as an element of $\Gal(K/F)$ (lemma \ref{lem:liftingauts}). Let $d\in L$, we need to show that $\sigma(d)\in L\dgen{\CC}$: Consider $\tau\in \Gal(K/L)$ and observe that, as $\Gal(K/L)$ is normal in $\Gal(K/F)$, we have that \[\bar{\sigma}^{-1}\tau\bar{\sigma}(d)=\sigma^{-1}\tau\sigma(d)=d.\] That is, $\tau(\sigma(d))=\sigma(d)$. As $\tau$ is arbitrary, this implies that $\sigma(d)$ belongs to the set fixed by $\Gal(K/L)$, which is precisely $L\dgen{\CC}$.

Finally, consider the restriction map \[|_{L\dgen{\CC}}\colon \Gal(K/F)\to \Gal(L/F).\] This map is onto because of saturation of $\UU$ and quantifier elimination, and its kernel is $\Gal(K/L)$. This implies that  $G(\CC)/G_L\cong \Gal(L/F)$.
\end{proof}

\begin{theorem}\label{thm:pet}
Let $K/F$ be a strongly normal extension of ID-fields. Suppose $F$ is relatively algebraically closed in $\UU$. Then, there is $\alpha\in G(K)$ such that $K=F\dgen{\alpha}$, and for all $\sigma\in \Gal(K/F)$, we have that 
\[\sigma(\alpha)=(\mu(\sigma))^{-1}\cdot \alpha.\]
\end{theorem}
\begin{proof}
Let $b,c\in \UU$ with $\tp(b/F)=\tp(a/F)=\tp(c/F)$ such that $a \ind_F b$,
$a \ind_F c$ and $c \ind_F b$. Using the notation from the proof of \ref{fact:definability}, it is clear that
\[h(a,b),\  h(a,c),\ h(c,b)\in G(\UU)\]
and
\[h(a,c)\cdot h(c,b)=h(a,b).\]
Replacing $a, b, c$ by $a, b, c$ plus finitely many derivations and $p$th-roots, we may assume $h$ is rational. Since $b$ is algebraically independent of $a$ and $c$ over $F$ and $F$ is relatively algebraically closed in $\UU$, we can find $d\in F$ such that
\[h(a,d),\ h(c,d),\ h(a,c)\in G(\UU),\]
and
\[h(a,c)\cdot h(c,d)=h(a,d).\hspace{1cm}(\star)\] 
Let $\alpha=h(a,d)$. Observe first that $\alpha$ is interdefinable
with $a$ over $F$ (because of the way $h$ is defined) and so
$K=F\langle\alpha\rangle$. Additionally, as $a,d\in K$, we have that $\alpha\in
K$. Finally, let $\sigma\in Gal(K/F)$ and pick
$c\ind_F\sigma(a)$. Then $\tp(a, c/F)=\tp(c, \sigma(a)/F)$ by stationarity. So 
\[h(c,\sigma(a))\cdot h(\sigma(a),d)=h(c,d).\]
But this, combined with $(\star)$ implies that
\[h(a,c)\cdot h(c,\sigma(a))\cdot h(\sigma(a),d)=h(a,d).\]
Now, $h(a,c)\cdot h(c,\sigma(a))=\mu(\sigma)$ and
$h(\sigma(a),d)=\sigma(\alpha)$ (because $\sigma(d)=d$). So, we have
\[\mu(\sigma)\cdot \sigma(\alpha)=\alpha.\] 
Which is what was left to prove.
\end{proof}

\section{Iterative differential Galois extensions}\label{Section:Tres}

In the previous section we introduced a class of ID-field extensions with well behaved Galois groups. Now, we will concentrate on the differential equations that under suitable conditions have good Galois theory. 

\subsection{Arc bundles and the iterative logarithmic derivative}

\begin{definition}
For a natural number $m$ and an arbitrary field $k$, define $k^{(m)}$ as the ring $k[\epsilon]/(\epsilon^{m+1})$. View $k^{(m)}$ as a $k$-algebra under the natural map $a\mapsto a+0\epsilon+\cdots+0\epsilon^{m}$. \end{definition}

\begin{definition}

Let $X$ be an algebraic variety over $F$ and define, for any field $k$ extending $F$ \textbf{the $m^{th}$ arc bundle of $X$ over $k$}, denoted $\AA_{m}X(k)$, as the set of $k^{(m)}$-rational points of $X$. This can be seen as an actual algebraic variety by identifying points in $k^{(m)}$ with points in $(k^{m+1})$.

Now, if $f\colon X\to Y$ a (regular) map of algebraic varieties over $F$, then define $$\AA_{m}(f)\colon \AA_{m}X\to \AA_{m}Y$$ as the map which is given, on $k$-points, by evaluating $f$ on $X(k^{(m)})$. \end{definition}

\begin{fact}Let $X, Y$ and $Z$ be algebraic varieties over $F$. 
\begin{enumerate}
\item $\AA_{m}(X)\times \AA_{m}(Y)$ is naturally isomorphic to $\AA_{m}(X\times Y)$.
\item Suppose $f\colon X\to Y$ and $g\colon Y\to Z$ are regular maps defined over $F$, then $\AA_{m}(g\circ f)=\AA_{m}(g)\circ\AA_{m}(f)$. This is, $\AA_{m}$ is a functor from the category of algebraic varieties with regular maps over $F$ to itself. 
\item If $(G,\ \cdot\ )$ is an algebraic group defined over $F$, then so is $(\AA_{m}G,\AA_{m}(\ \cdot\ ))$.
\end{enumerate}
\end{fact}
\begin{proof}
(1) and (2) are immediate consequences of the given definition of $\AA$, and (3) follows from those two. For instance, for associativity, consider the commutative diagram,\[\xymatrix{(g_{1},g_{2},g_{3})\ar@{|->}[r]\ar@{|->}[d] &(g_{1}\cdot g_{2},g_{3})\ar@{|->}[d]\\
(g_{1},g_{2}\cdot g_{3})\ar@{|->}[r] & (g_{1}\cdot g_{2})\cdot g_{3}=g_{1}\cdot (g_{2}\cdot g_{3})}\]
and then apply $\AA_{m}$. 
\end{proof}

For $n>m$, the quotient map $k^{(n)}\to k^{(m)}$ induces a projection \[\rho_{n,m}\colon \AA_{n}X\to\AA_{m}X.\] Identifying $\AA_{0}$ with the identity functor we will write $\rho_{n,0}$ as $\rho_{n}$. 

\begin{definition}\label{def:arcs}

For $a\in X(k)$, \textbf{the $n^{th}$ arc space $\AA_{n}X_{a}$ of $X$ at $a$} is defined as the fibre of $\rho_{n}\colon \AA_{n}X\to X$ over $a$. 

Additionally, define $\AA X(k)$, \textbf{the full arc bundle of $X$ over $k$}, as the inverse limit of $(\AA_{i}X(k))_{i\in\omega}$. 

Observe that $\AA X(k)$ can be identified with the $k[[\epsilon]]$-rational points of $X$.\end{definition}

Let us consider now the case when $F$ is a $ID$-field of positive characteristic. As before, assume $\UU$ is a highly saturated model of $SCH_{p,1}$.

\begin{lemma} Let $X$ be an algebraic variety defined over $C_{F}$.  

If $a\in X(\UU)$, then $\nabla_{X}(a)=(\dd_{0}(a),\dd_{1}(a),\ldots)\in\AA X(\UU)$ and in particular $\nabla_{X, m}(a)=(\dd_{0}(a),\ldots,\dd_{m}(a))\in\AA_{m}X(\UU)$ for any $m$.

Additionally, if $Y$ is an algebraic variety, $f\colon X\to Y$ is a morphism and both are also defined over $C_{F}$, then $\AA(f)\circ\nabla_{X}=\nabla_{Y} \circ f$. 
\end{lemma}
\begin{proof} Before starting, observe that $\nabla_{X}(a)$ is just another way of presenting $\DD_{\dd}(a)$ once you identify, as suggested above, $\AA X(\UU)$ and  $X(\UU[[\epsilon]])$. 

For the first part, recall that $\DD_{\dd}\colon \UU\to\UU[[\epsilon]]\colon x\to \sum_{i=0}^{\infty}\dd_{i}(x)\epsilon^{i}$ is a ring homorphism. Then, working locally, if $p(x)$ is one of the defining polynomials of $X$ and $a\in X(\UU)$, then $p(\DD_{\dd}(a))=0$. Which is another way of saying that $\nabla_{X}(a)\in\AA X(\UU)$.

For the second part, note that $\DD_{\dd}$ is not only a ring homomorphism but a $\CC$-algebra homomorphism. Thus, again locally, if $q(x)$ is a polynomial with constant coefficients, then $\DD_{\dd}(q(x))=q(\DD_{\dd}(x))$.  
\end{proof}
\begin{corollary} If $(G, \cdot)$ is an algebraic group defined over the constants of $F$, then $(\AA G,\AA (\cdot))$ is also a group and $\nabla_{G}\colon G\to \AA(G)$ is a group embedding.
\end{corollary}
\begin{proof} Since $(G, \ \cdot\ )$ is an inverse limit of algebraic groups, it is a pro-algebraic group. For the second part, the previous lemma gives us that $\nabla_{G}\circ\ \cdot\ = \AA(\ \cdot\ )\circ \nabla_{G\times G}$.\end{proof}

Let $G$ be an algebraic group defined over $C_{F}$ and consider the following exact sequence of groups:

\[\xymatrix{\{(e,0)\}\ar@{->}[r] &\AA_{e}(G)\ar@{->}[r]^{i} & \AA(G)\ar@{->}[r]^{\pi} &G\ar@{->}[r] & \{e\},}\]where $i$ is the natural inclusion and $\pi$ the canonical projection. Note that $s\colon G\to \AA(G)\colon g\mapsto (g,0,0,\ldots)$ is a group embedding and so a homomorphic section of $\pi$. Recall also that the existence of such a homomorphic section provides us with an isomorphism $\AA(G)\cong \AA_{e}(G)\rtimes G$. Thus, we can identify $G$ with its image under $s$. Let $h\colon \AA(G)\to \AA_{e}(G)$ be the projection induced by this isomorphism. Given $(g,u)\in \AA(G)$, we have that $(g,u)=((g,u)\cdot s(g^{-1}))\cdot s(g)$, so $h((g,u))=(g,u)\cdot s(g^{-1})$. 
 
Although $h$ is not a group homomorphism, we have:
\begin{fact}\label{fact:key}If $h((g,u))=h((l,v))$ then $h((g,u)^{-1}\cdot(l,v))=(e,0)$. 
\end{fact}
\begin{proof}
Since $h((g,u))=h((l,v))$, then, by definition, \[(g,u)\cdot s(g^{-1})=(l,v)\cdot s(l^{-1}).\] Reorganizing the equation, we get \[(g,u)^{-1}\cdot(l,v)=s(g^{-1})\cdot s(l).\] Now, applying $h$ to both sides, we obtain \[h((g,u)^{-1}\cdot(l,v))=h(s(g^{-1}l))=s(g^{-1}l)\cdot s(l^{-1}g)=(e,0).\] \end{proof}

\begin{definition}\label{def:logder}
Define \textbf{the iterative logarithmic derivative} be the map \[\ell D \colon G(\UU)\to \AA G_{e}(\UU)\colon g\mapsto h(\nabla(g)).\]
\end{definition}

\begin{fact}\label{fact:KerlD} If $G$ is defined over $C_{F}$, then $Ker(\ell D)=G(\CC)$. Furthermore, if $\ell D(x)=\ell D(y)$, then $x^{-1}\cdot y\in G(\CC)$. \end{fact}
\begin{proof}
If $\ell D(g)=(e,0)$, then $\nabla(g)\cdot (g^{-1},0)=(e,0)$. Thus $\nabla(g)=(g,0)$, which implies that  $\dd_{i}(g)=0$ for any $i$. That is, $g\in G(\CC)$. 

The additional remark is a direct consequence of fact \ref{fact:key}. 
\end{proof}

The logarithmic derivative in the (differential) characteristic zero case, defined as a map from $G$ to $\LL(G)$, the Lie algebra of $G$, is surjective. In our setting, that is not the case:
\begin{example}
Let $G=\mathbb{G}_{a}$, the additive group. Then $\AA(G)=\sum_{i=0}^{\infty} \mathbb{G}_{a}$ and \[\ell D(g)=(g,\dd_{1}(g),\ldots)-(g,0,\cdots)=(0,\dd_{1}(g), \dd_{2}(g),\ldots).\] Thus, $\Imag(\ell D)$ is contained in the set \[\{(x_{i})\colon x_{0}=0 \text{ and, for } i>0, \dd_{j}(x_{i})=\binom{i+j}{i}(x_{i+j})\},\] which is clearly not equal to $\AA_{e}(G)$.
\end{example}

\subsection{Logarithmic differential equations and Galois extensions}

\begin{definition}\label{def:idge}Given an algebraic group defined over the constants of $(F,\dd)$ a (non-trivial) definably closed $ID$-field of characteristic $p$ with algebraically closed constant field, by a \textbf{consistent logarithmic differential equation} over $F$ we mean something of the form \[ \ell D(x)=\alpha,\] where $\ell D$ is defined as in the previous section and $\alpha \in \AA G_{e}(F)$ is an element contained in the image of $\ell D$. 

By an \textbf{iterative differential Galois extension} of $F$ for that given logarithmic differential equation we mean $K=F\dgen{a}^{s}$, where $\ell D(a)=\alpha$ and $C_{F}=C_{K}$. \end{definition}

\begin{theorem}[Existence and Uniqueness of iterative differential Galois extensions]If $G$ is an algebraic group defined over the constants of $(F,\dd)$ and $\ell D(x)=\alpha$ is a (consistent) logarithmic differential equation over $F$, then there exists an iterative differential Galois extension of $F$ for the given equation. Furthermore, any two such extensions are isomorphic over $F$ as $ID$-fields.\label{th:existanduniq}
\end{theorem}

Once again, we will depend on the use of an appropriate auxiliary structure in order to prove this. Let $\MM$ be the two-sorted structure $(\XX,\CC)$, where $\XX$ is the set of solutions in $\UU$ of the equation $\ell D(x)=\alpha$, and the relations of $\MM$ are those induced by $F$-definable sets in $\UU$. 

\begin{lemma}\label{lem:scaf2}
$\MM$ is saturated, its theory $\Th(\MM)$ has quantifier elimination, it is totally transcendental and, additionally,  \[\MM_{0}\cap \CC=C_{F}.\]
\end{lemma}

\begin{proof} Just as in the proof of fact \ref{fact:M}, this depends on the bi-interpretability of an expansion of $\MM$ by a constant and another simpler structure. Let $a'$ be {\em any} solution of the given logarithmic differential equation and let $\NN$ be the structure whose universe is $\CC$ and whose relations are induced by the $F\dgen{a'}$-definable sets in $\UU$. We will see that $\NN$ is bi-interpretable with $(\MM,a')$:

Let $Y=G(\CC)$. As a subset of $\NN$, we have that $Y$ is definable. Observe that there is a one-to-one correspondence between $\XX$ and $Y$. This is given by the fact that, for any $b\in \XX$, there is $g\in G(\CC)$ such that $g\cdot a'=b$ (a corollary of fact \ref{fact:key}). Note that such $g$ is unique given $b$ and $a'$. Let $f\colon Y\to \XX\colon g\mapsto g\cdot a'$. The fact that $\XX$ and $Y$ are isomorphic via this function is proved as in fact \ref{fact:biinterp}. This shows that $\XX$ (and so $\MM$) is interpretable in $\NN$. The fact that $\NN$ is interpreted in $(\MM, a)$ is clear. 

Note that $\NN$ is, once again, totally transcendental and saturated, and thus the argument used to prove fact \ref{fact:M} applies. Hence, $\MM$ is also saturated and totally transcendental. The same is true for proving that $\MM$ has quantifier elimination. 

The proof that $\MM_{0}\cap \CC=C_{F}$ is exactly the same given for lemma \ref{lem:primeM}. 
\end{proof}

We now go back to the proof of theorem \ref{th:existanduniq}:

\begin{proof}[Proof of theorem \ref{th:existanduniq}] ({\em Existence}) Let $a \in \MM_{0}\cap \XX$ and $K=F\dgen{a}^{s}$. It is not hard to see that $C_{F}=C_{K}$.




({\em Uniqueness}) Let $a'\in \XX$ be such that $K'=F\dgen{a'}^{s}$ is another iterative differential Galois extension of $F$ for the given equation. 

Consider $\MM_{1}$ prime over $a'$. Note that $\MM_{1}\cap \CC=C_{F}$. Let $a''\in M_{1}$ with the same type as $a$ over the empty set. As $\MM_{1}\prec\MM$, there is $g\in M_{1}\cap \CC=C_{F}$ such that $a''\cdot g=a'$. This implies that $\dcl(Fa')=\dcl(Fa'')$.

Finally, since $\tp(a)=\tp(a'')$ in $\MM$, then $\tp(a/F)=\tp(a''/F)$ in $\UU$. This, by saturation, implies that $\dcl(Fa)$ is isomorphic to $\dcl(Fa'')=\dcl(Fa')$.


 
\end{proof}

\section{What goes around comes around}\label{Section:cuatro}

Section \ref{Section:Dos} introduced a class of extensions of differential fields with good Galois theory. Section \ref{Section:Tres} provided us with extensions of differential fields related to iterative logarithmic differential equations. In this section we will prove that, under certain conditions on the base field, these two notions coincide. 

\begin{theorem}\label{th:galextstrnor}
If $K$ is an iterative differential Galois extension of $F$ for a given logarithmic differential equation $\ell D(x)=\alpha$, then $K/F$ is a strongly normal extension.  
\end{theorem}

\begin{proof}
The first two conditions of the definition of strongly normal extensions are explicitly stated in our definition of $ID$-Galois extensions. For the third one, let $K=F\dgen{a}^{s}$, and $\tp(a'/F)=\tp(a/F)$. Since $\alpha\in \AA G_{e}(F)$ and $lD(a)=\alpha$, we have that $\ell D(a')=\alpha$, and this implies that $a^{-1}a'\in G(\CC)$ by fact \ref{fact:KerlD}. So, $a'=a\cdot d$ for some $d\in  G(\CC)$ and thus $a'\in K\dgen{\CC}$. Finally, for the fourth condition, the argument goes exactly as in the proof of the first part of theorem \ref{th:galcorr}. 
\end{proof}

\begin{theorem}\label{th:equiv}Suppose $F$ is relatively algebraically closed in $\UU$ and $K/F$ is a strongly normal extension. Let $G$ be the algebraic group over $C_{F}$ that is provided by theorem  \ref{fact:definability} whose set of $\CC$-rational points is isomorphic to $\Gal(K/F)$. Then $K/F$ is an iterative differential Galois extension for some logarithmic differential equation on $G$.
\end{theorem}

\begin{proof}
Let $K/F$ be a strongly normal extension and $G$ as in the statement. Let \[\mu\colon \Gal(K/F)\to G(\CC)\] witness the isomorphism. 

Since $F$ is relatively algebraically closed in $\UU$, the primitive element theorem provides us with $a \in G(K)$ such that $K=F\dgen{a}$ and, for any $\sigma\in \Gal(K/F)$, we have that $\sigma(a)=(\mu(\sigma))^{-1}\cdot a$. 

Let $\alpha=\ell D(a)$ and note that $\alpha$ is fixed by any automorphism of $\UU$ fixing $F$: let $\bar{\xi}\in Aut(\UU/F)$; since $K/F$ is strongly normal, lemma \ref{lem:liftingauts} tells us that $\xi=\bar{\xi}|_{K\dgen{\CC}}\in \Gal(K/F)$, and so $\xi(a)\cdot a^{-1}=(\mu(\xi))^{-1}\in G(\CC)=Ker(\ell D)$. Thus, \[\bar{\xi}(\alpha)=\ell D(\xi(a))=\ell D(a)=\alpha.\] Since $F$ is relatively algebraically closed in $\UU$, we get that $\alpha \in \AA G_{e}(F)$. 

Consider the iterative logarithmic differential equation $\ell D(x)=\alpha$. Let $K'=\dcl(Fa')$ be the {\em unique} iterative differential Galois extension of $F$ for the given equation. Note that, since $\ell D(a)=\ell D(a')$, fact \ref{fact:KerlD} tells us that there exists $g\in G(\CC)$ such that $a'=g^{-1}\cdot a$. Since $g\in G(\CC)$ there is $\sigma\in \Gal(K/F)$ such that $\mu(\sigma)=g$. So, by the way $a$ was chosen, \[a'=((\mu(\sigma))^{-1}\cdot a=\sigma(a),\]
which implies that $\sigma$ induces an isomorphism between $K$ and $K'$. Thus $K$ is isomorphic to a Galois differential extension of $F$ for an appropriate logarithmic differential equation.\end{proof}

To conclude, let us prove, under the same assumption on the base field, that a desirable equality between the transcendence degree of a strongly normal extension and the dimension of its Galois group holds. To be precise:

\begin{theorem}\label{th:trdegdim}
Suppose $F$ is relatively algebraically closed in $\UU$, the extension $K/F$ is strongly normal, and $G$ is an algebferaic group over $C_{F}$ such that $G(\CC)$ is isomorphic to $\Gal(K/F)$. Then,
\[\dim(G(\CC))=\trdeg(K/F).\]
\end{theorem} 

\begin{proof}
By the previous result, we know that $K$ is an iterative differential Galois extension of $F$ for a consistent logarithmic differential equation $\ell D(x)=\alpha$ on $G(\UU)$ with $\alpha$ in $F$. That is, there is $a\in G(K)$ such that $\ell D(a)=\alpha$ and $K=F\dgen{a}^{s}$. 

First note that $K=F(a)^{s}$. To prove this, it is enough to see that $\dd_{n}(a)\in F(a)$ for all $n\in \omega$. However, by the definition of the logarithmic derivative, we know that $\nabla(a)=\alpha \cdot (a,0,0,\ldots)$, thus coordinate by coordinate, we obtain that $\dd_{n}(a)$ is a rational function of $a$ and the coefficients of $\alpha$, which are all in $F$. This in particular implies that $\trdeg(K/F)=\trdeg(F(a)/F)$.

Secondly, by fact \ref{fact:KerlD}, we know that for each $a', a''\in \XX$, where $\XX$ is the solution set in $G(\UU)$ of the equation, there is $g\in G(\CC)$ such that $g\cdot a'=a''$. Since $G(\CC)$ is precisely our isomorphic copy of $\Gal(K/F)$, this implies that, in $\UU$, any $a'\in\XX$ has the same type as $a$ over $F$. Thus, $\trdeg(K/F)$ is in fact equal to $\trdeg(\XX)$.

Finally, observe that, after naming $a\in \XX$, there is a rational bijection between $\XX$ and $G(\CC)$ given by $a'\mapsto a^{-1}\cdot a'$. So, $\trdeg(K/F)=\trdeg(\XX)=\trdeg(G(\CC))=\dim(G(\CC))$. 
\end{proof}



\end{document}